\documentclass[sn-mathphys,Numbered]{sn-jnl}

\usepackage{graphicx}%
\usepackage{graphics}
\usepackage{subfigure}
\usepackage{multirow}%
\usepackage{amsmath,amssymb,amsfonts}%
\usepackage{amsthm}%
\usepackage{mathrsfs}%
\usepackage[title]{appendix}%
\usepackage{xcolor}%
\usepackage{textcomp}%
\usepackage{manyfoot}%
\usepackage{booktabs}%
\usepackage{algorithm}%
\usepackage{algorithmicx}%
\usepackage{algpseudocode}%
\usepackage{listings}%

\allowdisplaybreaks[4]

\newcommand{\inner}[1]{\left\langle #1 \right\rangle}
\newcommand{\norm}[1]{\left\Vert #1\right\Vert}
\newcommand{\bb}[1]{\mathbb{#1}}

\newcommand{\ca}[1]{\mathcal{#1}}

\newcommand{\tp}{^\top}

\newcommand{\prox}{{\mathrm{prox}}}

\newcommand{\xk}{{x_{k} }}

\newcommand{\xkp}{{x_{k+1} }}

\newcommand{\Gc}{\ca{G}}
\newcommand{\Pc}{\ca{P}}

\newcommand{\Hc}{\ca{H}}

\newcommand{\R}{\mathbb{R}}
\newcommand{\Rn}{\mathbb{R}^n}

\newcommand{\Rm}{\mathbb{R}^m}

\newcommand{\Rmn}{\mathbb{R}^{m\times n}}

\DeclareMathOperator*{\argmin}{arg\,min}

\theoremstyle{thmstyleone}%
\newtheorem{theo}{Theorem}
\newtheorem{prop}[theo]{Proposition}%
\newtheorem{lem}[theo]{Lemma}
\newtheorem{coro}[theo]{Corollary}
\newtheorem{defin}[theo]{Definition}
\newtheorem{assumpt}[theo]{Assumption}
\theoremstyle{thmstyletwo}%

\theoremstyle{thmstylethree}%

\begin{document}

\title[IPZOPM]{An Inexact Preconditioned Zeroth-order Proximal Method for Composite Optimization}


\author[1,2]{\fnm{Shanglin} \sur{Liu}}\email{\href{mailto:liushanglin@lsec.cc.ac.cn}{liushanglin@lsec.cc.ac.cn}}
\equalcont{These authors contributed equally to this work.}

\author*[3]{\fnm{Lei} \sur{Wang}}\email{\href{mailto:wlkings@lsec.cc.ac.cn}{wlkings@lsec.cc.ac.cn}}
\equalcont{These authors contributed equally to this work.}

\author[4]{\fnm{Nachuan} \sur{Xiao}}\email{\href{mailto:xnc@lsec.cc.ac.cn}{xnc@lsec.cc.ac.cn}}
\equalcont{These authors contributed equally to this work.}

\author[1,2]{\fnm{Xin} \sur{Liu}}\email{\href{mailto:liuxin@lsec.cc.ac.cn}{liuxin@lsec.cc.ac.cn}}

\affil[1]{\orgdiv{State Key Laboratory of Scientific and Engineering Computing}, \orgname{Academy of Mathematics and Systems Science, Chinese Academy of Sciences}, \orgaddress{\city{Beijing}, \postcode{100190}, \country{China}}}

\affil[2]{\orgdiv{School of Mathematical Sciences}, \orgname{University of Chinese Academy of Sciences}, \orgaddress{\city{Beijing}, \postcode{101408}, \country{China}}}

\affil[3]{\orgdiv{Department of Statistics}, \orgname{Pennsylvania State University}, \orgaddress{\city{University Park}, \postcode{16802}, \country{USA}}}

\affil[4]{\orgdiv{Institute of Operational Research and Analytics}, \orgname{National University of Singapore}, \orgaddress{\city{Singapore}, \postcode{117597}, \country{Singapore}}}


\abstract{In this paper, we consider the composite optimization problem, where the objective function integrates a continuously differentiable loss function with a nonsmooth regularization term. Moreover, only the function values for the differentiable part of the objective function are available. To efficiently solve this composite optimization problem, we propose a preconditioned zeroth-order proximal gradient method in which the gradients and preconditioners are estimated by finite-difference schemes based on the function values at the same trial points. We establish the global convergence and worst-case complexity for our proposed method. Numerical experiments exhibit the superiority of our developed method.}




\maketitle

\section{Introduction}\label{sec1}

Zeroth-order optimization, commonly referred to as derivative-free or black-box optimization \cite{powell2002uobyqa,powell2008developments,conn2009introduction11,zhang2012derivative,grapiglia2016derivative,larson2019derivative}, 
finds important applications across diverse domains, particularly in cases 
where either the objective function remains implicit or its gradients are impossible or prohibitively expensive to evaluate. 
These applications include circuit design \cite{ciccazzo2015derivative}, 
structured prediction \cite{taskar2005learning}, 
computational nuclear physics \cite{wild2015derivative}, 
bandit learning \cite{shamir2017optimal}, 
optimal parameter determination in material science experiments \cite{nakamura2017design}, 
neural network adversarial attacks \cite{papernot2017practical}, 
and hyper-parameter tuning \cite{snoek2012practical}.

In this paper, we aim to propose a zeroth-order proximal gradient method for the following optimization problem with a composite structure,
\begin{equation}
    \label{opt:dfo}
    \min_{x \in \Rn}\quad h(x):= f(x) + r(x). 
\end{equation}
Here, the loss function $f: \Rn \to \R$ is  continuously differentiable and possibly nonsmooth, where only the evaluations of its function values are available. 
Moreover, the regularization term $r: \Rn \to \R \cup \{\infty\}$ is an explicit convex extended real-valued function (possibly nonsmooth). 
The inclusion of the regularization term $r$ enables the explicit utilization of prior knowledge concerning the problem structure. 
Throughout this paper, we make the following assumptions on the problem \eqref{opt:dfo}.

\begin{assumpt}
    \label{asp:function}

    \mbox{}
    
    \begin{enumerate}

        \item The objective function $h: \Rn \to \R \cup \{\infty\}$ is bounded from below, namely, there exists a constant $\underline{h}$ such that $h(x) \geq \underline{h}$ for any $x \in \Rn$. 
    
        \item The loss function $f: \Rn \to \R$ is twice continuously differentiable with $L_f$-Lipschitz continuous gradients and $M_f$-Lipschitz continuous Hessian.
        
        \item The regularization term $r: \Rn \to \R \cup \{\infty\}$ is convex, proper and lower-semicontinuous. 
        
        \item The proximal mapping $\mathrm{prox}_{\lambda r} (x)$ of $r$, 
        which is defined by
        \begin{equation*}
            \mathrm{prox}_{\lambda r} (x) := \argmin_{y \in \Rn} \; r(y) + \dfrac{1}{2\lambda} \norm{y - x}^2,
        \end{equation*}
        is easy-to-compute for any $\lambda > 0$ and $x \in \bb{R}^n$. 
    \end{enumerate}
\end{assumpt}

It is important to highlight that while we assume the Lipschitz continuity of $\nabla f$ and $\nabla^2 f$, we do not presume their accessibility as per Assumption \ref{asp:function}.

\subsection{Existing Works}\label{subsec1}

Extensive research has been dedicated to zeroth-order optimization over the past decades. 
For an indepth overview of these progresses, interested readers are referred to a comprehensive survey \cite{larson2019derivative} and an open-source software PDFO \cite{ragonneau2023pdfo} for details. 
In this subsection, we provide a concise overview of a select subset of studies, focusing on those most closely related methods to the present topic. 

Given the composite structure, the classical proximal gradient method \cite{mine1981minimization} emerges as a natural choice to solve the problem \eqref{opt:dfo} via the following proximal subproblem,
\begin{equation*}
    \xkp = \argmin_{y \in \Rn} ~ \inner{\nabla f(\xk), y - \xk} + \dfrac{1}{2\eta_k} \norm{y - \xk}^2 + r(y),
\end{equation*}
where $\eta_k > 0$ is a stepsize.
However, it is evident that the computation of gradients is inherent in each iteration of the aforementioned algorithm, 
rendering it unsuitable for the zeroth-order setting considered in this paper. 
To address this issue, various existing algorithms propose different zeroth-order schemes to approximate the gradients by function values, such as the Gaussian smoothing scheme \cite{ghadimi2016mini,nesterov2017random}
and finite-difference scheme \cite{gu2018faster,huang2019faster}.

Recently, Kalogerias and Powell \cite{kalogerias2022zeroth} introduced and analyzed a zeroth-order algorithm based on the Gaussian smoothing technique, 
with a specific focus on applications in risk-averse learning. 
In their work, the regularization term $r$ is selected as an indicator function of a closed convex set. 
Balasubramanian and Ghadimi \cite{balasubramanian2022zeroth}, building on an earlier work presented in \cite{ghadimi2013stochastic}, 
explored zeroth-order methods specialized for nonconvex stochastic optimization problems. 
Their approach assumes that $r$ is an indicator function and emphasizes addressing high-dimensionality challenges while avoiding saddle points.

For the general case, Kungurtsev and Rinaldi \cite{kungurtsev2021zeroth} leveraged a double Gaussian smoothing scheme within zeroth-order proximal gradient methods to solve the nonconvex and nonsmooth problem \eqref{opt:dfo}.
This algorithm was later improved in \cite{pougkakiotis2022zeroth} by adopting a more streamlined single Gaussian smoothing scheme.
A similar approach, utilizing different zeroth-order oracles, was proposed and analyzed in \cite{cai2022zeroth} under a novel approximately sparse gradient assumption.
To enhance the convergence rate in the stochastic setting, Huang et al. \cite{huang2019faster} incorporated the variance reduction technique \cite{xiao2014proximal,defazio2014saga} into the development of zeroth-order algorithms. 
Very recently, Doikov and Grapiglia \cite{doikov2023zerothorder} developed a zeroth-order implementation of the cubic Newton method \cite{nesterov2006cubic} to improve the global complexity.
However, it is noteworthy that this algorithm approximates the full Hessian matrix with additional function-value evaluations per iteration and the resulting subproblem is not easy to solve in general.

\subsection{Motivation}\label{subsec:motivation}

In this paper, we focus on the finite-difference scheme used to estimate $\nabla f(x)$ at the trial points $\{x \pm \delta e_i: 1\leq i\leq n\}$ as follows,
\begin{equation}
\label{eq:grad}
    \Gc_{\delta} (x) := \sum_{i = 1}^n \frac{f(x + \delta e_i) - f(x - \delta e_i)}{2\delta} e_i.
\end{equation}
where $e_i \in \Rn$ denotes the $i$-th column of the $n \times n$  identity matrix $I_n$, and $\delta > 0$ is a coefficient controlling the approximation error.
As a result, typical zeroth-order proximal gradient (ZOPG) methods \cite{pougkakiotis2022zeroth} generate the next iterate by solving the following subproblem,
\begin{equation*}
    \xkp = \argmin_{y \in \Rn} ~ \inner{\Gc_{\delta} (\xk), y - \xk} + \dfrac{1}{2\eta_k} \norm{y - \xk}^2 + r(y).
\end{equation*}
It is worth mentioning that the above algorithm, albeit easy to implement, are plagued by the slow convergence rate since only the first-order information is employed. 
Consequently, it requires to evaluate the function values at a large number of trial points to reach a certain accuracy, which gives rise to high computational costs.

To address this issue, we aim to utilize the second-order information of $f$ without additional evaluations of function values.
Specifically, based on the trial points $\{x \pm \delta e_i: 1\leq i\leq n\}$, the diagonal entries of $\nabla^2 f(x)$ can be readily estimated through the following finite-difference scheme, 
\begin{equation}
\label{eq:hess}
    \Hc_{\delta} (x) := \sum_{i = 1}^n \frac{f(x + \delta e_i) + f(x - \delta e_i) - 2f(x)}{\delta^2} e_i e_i\tp,
\end{equation}
where $\delta > 0$ is a parameter for controlling the approximation error.
Therefore, to incorporate the second-order information of $f$ into our algorithm, we introduce the following update scheme for $\xkp$, 
\begin{equation}
\label{opt:prox_sub}
    \xkp \approx \mathop{\arg\min}_{y \in \Rn} ~ \inner{\Gc_{\delta_k}(\xk), y - \xk} + \dfrac{1}{2} \inner{ \left( \Hc_{\delta_k} (\xk) + \sigma_k I_n \right) (y - \xk), y - \xk} + r(y),
\end{equation}
where $\delta_k > 0$ and $\sigma_k > 0$ are two constants. 
We will discuss how to solve the above subproblem in details in the sequel.

\subsection{Contribution}

In this paper, we propose a novel zeroth-order proximal method for solving the optimization problem \eqref{opt:dfo}. 
In our approach, we employ $\Gc_{\delta}(x)$ as the estimation for the first-order information, and construct the preconditioner $\Hc_{\delta}(x)$ to capture the second-order information from the same set of trial points.
Then the iterate $\xk$ is updated by inexactly solving the proximal subproblem \eqref{opt:prox_sub}. 
To the best of our knowledge, this is the first derivative-free algorithm that approximates the diagonals of Hessian matrices for composite optimization problems.
And the global convergence is first established with a worst-case complexity.
Preliminary numerical experiments demonstrate the superior performance of our method compared to existing zeroth-order methods. 

\subsection{Organization}

The remainder of this paper is structured as follows. 
In Section \ref{sec:algorithm}, we present an inexact preconditioned zeroth-order algorithm to solve the composite optimization problem \eqref{opt:dfo}. 
Additionally, the convergence properties of the proposed algorithm are thoroughly examined in Section \ref{sec:convergence}. 
Section \ref{sec:numerical} showcases numerical experiments conducted on a range of test problems, providing insights into the performance of the proposed algorithm. 
The final section offers concluding remarks and discusses potential avenues for future developments.

\section{An Inexact Preconditioned Zeroth-order Proximal Method}
\label{sec:algorithm}

This section develops an inexact preconditioned zeroth-order proximal method for the nonconvex and nonsmooth optimization problem \eqref{opt:dfo}.

\subsection{Basic notations and preliminaries}

Throughout this paper, we consider a Euclidean space $\Rn$ endowed with an inner product $\inner{\cdot, \cdot}$ and the induced norm $\norm{x} = \inner{x, x}^{1 / 2}$.
For any function $\phi: \Rn \to \R \cup \{\infty\}$, the domain is defined as the following set 
\begin{equation*}
    \mathrm{dom}(\phi) := \{x \in \Rn: \phi(x) < \infty\}.
\end{equation*} 
We say that $\phi$ is proper if $\mathrm{dom}(\phi) \neq \emptyset$.

The subdifferential of a function $\phi$ at $x \in \mathrm{dom}(\phi)$, denoted by $\partial \phi (x)$, consists of all vectors $v \in \Rn$ satisfying 
\begin{equation*}
    \phi (y) \geq \phi (x) +\inner{v, y - x} + o(\norm{y - x}), \quad
    \text{as~} y \to x.
\end{equation*}
We set $\partial \phi (x) = \emptyset$ for all $x \notin \mathrm{dom}(\phi)$.
When $\phi$ is smooth, the subdifferential $\partial \phi(x)$ consists only of the gradient $\{\nabla \phi (x)\}$,
while for convex functions it reduces to the subdifferential in the sense of convex analysis \cite{rockafellar2015convex}.



\subsection{Stationarity condition}
\label{subsec:stationarity}


Similar to the smooth setting, the primary goal of nonsmooth nonconvex optimization is the search for stationary points.
Based on the notion of subdifferential, a point $x \in \Rn$ is called stationary for the problem \eqref{opt:dfo} if the following inclusion holds \cite{davis2019stochastic},
\begin{equation*}
    0 \in \partial h(x) = \nabla f(x) + \partial r(x).
\end{equation*}
However, it is usually difficult to monitor $\mathrm{dist} (0, \partial h(x))$ in practice to measure the progress of an algorithm.
In this subsection, we derive a surrogate measurement based on the following proximal-gradient mapping \cite{davis2019stochastic},
\begin{equation*}
    \Pc_{\gamma}(x) := \dfrac{1}{\gamma} \left(x - \mathrm{prox}_{\gamma r}(x - \gamma \nabla f(x))\right),
\end{equation*}
where $\gamma > 0$ is a constant.

\begin{lem} \label{le:prox_subdiff}
    Let Assumption \ref{asp:function} hold. 
    For any $\gamma > 0$, we denote $\hat{x} := \mathrm{prox}_{\gamma r}(x - \gamma \nabla f(x))$.
    Suppose that the point $x \in \Rn$ satisfies the following condition, 
    \begin{equation}
        \label{eq:prox_subdiff_0}
        \norm{\Pc_{\gamma}(x)} \leq \epsilon,
    \end{equation}
    where $\epsilon > 0$ is a small constant.
    Then it holds that 
    \begin{equation*}
        \mathrm{dist}(0, \partial h(\hat{x})) \leq (1 + L_f \gamma)\epsilon.
    \end{equation*} 
\end{lem}

\begin{proof}
    To begin with, the condition \eqref{eq:prox_subdiff_0} directly implies that 
    \begin{equation*}
        \norm{\hat{x} - x} \leq \gamma\epsilon.
    \end{equation*} 
    Moreover, from the definition of $\hat{x}$ and proximal mappings, we have
    \begin{equation*}
        \Pc_{\gamma}(x) \in \nabla f(x) + \partial r(\hat{x}),
    \end{equation*}
    which further infers that
    \begin{equation*}
        \mathrm{dist}(0, \nabla f(x) + \partial r(\hat{x})) \leq \epsilon.
    \end{equation*}
    Therefore, it can be straightforwardly verified that
    \begin{equation*}
        \mathrm{dist}(0, \partial h (\hat{x})) \leq \mathrm{dist}(0, \nabla f(x) + \partial r(\hat{x})) + L_f\norm{x - \hat{x}} \leq (1 + L_f \gamma)\epsilon.
    \end{equation*}
    This completes the proof. 
\end{proof}

The above lemma indicates that, if the norm $\norm{\Pc_{\gamma}(x)}$ is sufficiently small, $x$ is very close to the proximal-gradient point $\hat{x}$ that is nearly stationary for the problem \eqref{opt:dfo}.
In this sense, the quantity $\norm{\Pc_{\gamma}(x)}$ can be regarded as the measurement of the stationarity violation.
This observation motivates the following definition of $\epsilon$-stationary points of \eqref{opt:dfo}.

\begin{defin}
\label{def:stationary}
    We say that a point $x \in \Rn$ is an $\epsilon$-stationary points of \eqref{opt:dfo} if there exists a constant $\gamma_{\max} > 0$ such that $\norm{\Pc_{\gamma}(x)} \leq \epsilon$ for any $\gamma \in (0, \gamma_{\max})$.
\end{defin}


\subsection{Algorithm development}
\label{subsec:algorithm}

As mentioned in Section \ref{subsec:motivation}, the main goal of this paper is to employ the second-order information to improve the performance of zeroth-order proximal gradient methods.
However, it will incur high computational costs to construct the finite-difference approximation of the full Hessian by only using function values \cite{doikov2023zerothorder}.
Additionally, solving the resulting subproblem is often challenging, which does not admit a closed-form solution in general.
To cope with these issues, we propose an inexact preconditioned zeroth-order proximal method (IPZOPM).
Our algorithm only estimates the diagonal entries of the Hessian matrix based on the zeroth-order scheme \eqref{eq:hess}.
If the nonsmooth term $r$ possesses some favourable properties, the resulting subproblem will admit a closed-form solution.
For general cases, our algorithm takes advantage of an inexact strategy to solve the subproblem.

In our algorithm, we initiate by evaluating function values at the trial points $\{\xk \pm \delta_k e_i: i\in [n]\}$ around the current iterate $\xk$, where $\delta_k > 0$ is a constant.
Using these function values, we then compute $\Gc_{\delta_k}(\xk)$ through \eqref{eq:grad} as an approximation of $\nabla f(\xk)$.
Moreover, the diagonal of $\nabla^2 f(\xk)$ can be simultaneously estimated by \eqref{eq:hess} based on the same function values.
In other words, we are able to compute $\Hc_{\delta_k}(\xk)$ without additional evaluations of function values.

Subsequently, the approximate model of $f$ around the current iterate $\xk$ is assembled in the subproblem \eqref{opt:prox_sub} that captures a part of the second-order information.
For convenience, we denote
\begin{equation*}
    l(y; x, \delta, \sigma) := f (x) + \inner{\Gc_{\delta}(x), y - x} + \dfrac{1}{2} \inner{ \left( \Hc_{\delta} (x) + \sigma I_n \right) (y - x), y - x} + r(y).
\end{equation*}
To reduce computational overheads, the proposed algorithm only inexactly solves the subproblem \eqref{opt:prox_sub} to obtain the next iterate $\xkp$ in the following manner,
\begin{equation}
\label{eq:inexact}
    l(\xkp; \xk, \delta_k, \sigma_k) - \inf_{y \in \Rn} l(y; \xk, \delta_k, \sigma_k) \leq \varepsilon_k,
\end{equation}
where $\varepsilon_k > 0$ is the prescribed precision at iteration $k$.
It is worth mentioning that $l(y; x, \delta, \sigma)$ is a strongly convex function with respect to $y \in \Rn$. 
Employing the proximal gradient method to solve the subproblem \eqref{opt:prox_sub} ensures a linear convergence rate \cite{beck2017first}. 
Hence, it takes only $\ca{O}(\log(1 / \varepsilon_k))$ inner iterations to achieve a solution satisfying \eqref{eq:inexact}. 
Furthermore, when the proximal mapping of $r$ is further assumed to be semi-smooth, applying the semi-smooth Newton method to solve the subproblem \eqref{opt:prox_sub} guarantees the local superlinear convergence rate \cite{xiao2018regularized,li2018highly}. 
As a result, it is not difficult in practice to compute the next iterate $\xkp$ that satisfies \eqref{eq:inexact}.

Finally, it is essential to note that the nonsmooth term $r(x)$ is usually separable with respect to the entries of $x$ in many applications, namely, 
\begin{equation*}
    r(x) = \sum_{i = 1}^n r_i([x]_i),
\end{equation*}
where $[v]_i$ denotes the $i$-th entry of a vector $x$.
Given that $\Hc_{\delta_k}(\xk)$ is a diagonal matrix, the subproblem \eqref{opt:prox_sub} can be efficiently solved by the following coordinate-wise manner,
\begin{equation*}
    [\xkp]_i = \mathrm{prox}_{({1}/{\tau_{i, k}})r_i} \left([\xk]_i -  \dfrac{1}{\tau_{i, k}} [\Gc_{\delta_k}(\xk)]_i\right),
    \quad i \in [n],
\end{equation*}
where $\tau_{i, k} = [\Hc_{\delta_k}(\xk)]_{ii} + \sigma_k$
and $[A]_{ii}$ denotes the $(i, i)$-th entry of a matrix $A$.
Furthermore, if the proximal mapping of each $r_i$ has a closed-form solution, the subproblem \eqref{opt:prox_sub} can be solved exactly at a negligible cost.
This is notably the case for the $\ell_1$ regularizer $r(x) = \mu\norm{x}_1$, where $\mu > 0$ is a constant and $\norm{x}_1 := \sum_{i = 1}^n |[x]_i|$.

The whole procedure of our algorithm IPZOPM is summarized in Algorithm \ref{alg:IPZOPM}.
We can see that each iteration of IPZOPM only involves the evaluation of function values of $f$ on $2n$ points.

\begin{algorithm}
    \begin{algorithmic}[1]   
        \Require function $h = f + r$.
        \State Choose an initial guess $x_0 \in \Rn$.
        \State Set $k := 0$.
        \While{``not terminated''}
            \State Determine the parameters $\delta_k$, $\sigma_k$ and $\varepsilon_k$.
            \State Evaluate the function values on the trial points $\{\xk \pm \delta_k e_i: i\in [n]\}$. 
            \State Compute $\Gc_{\delta_k}(\xk)$ and $\Hc_{\delta_k}(\xk)$.
            \State Solve the subproblem \eqref{opt:prox_sub} inexactly such that $\xkp$ satisfies \eqref{eq:inexact}.
            \State Set $k := k + 1$.
        \EndWhile
        \State Return $\xk$.
    \end{algorithmic}  
    \caption{An inexact preconditioned zeroth-order proximal method (IPZOPM) for \eqref{opt:dfo}.}  
    \label{alg:IPZOPM}
\end{algorithm}

\section{Convergence Properties}
\label{sec:convergence}

In this section, we rigorously establish the global convergence of IPZOPM with the worst-case complexity.
To this end, we make the following mild assumptions on Algorithm \ref{alg:IPZOPM}.

\begin{assumpt}
    \label{asp:parameter}
    There exists a constant $L_H > 0$ such that $\sup_{k\geq 0} \norm{\Hc_{\delta} (\xk)} \leq L_H$. 
\end{assumpt}

We begin our proof by characterizing the errors in estimating the gradient and the diagonal of the Hessian in the following lemma. 

\begin{lem}
    \label{Le_diff_approx}
    Suppose Assumption \ref{asp:function} holds. Then it can be readily verified that
    \begin{equation*}
        \norm{\Gc_{\delta}(x) - \nabla f(x)} \leq \frac{M_f \delta^2}{2},
    \end{equation*}
    and
    \begin{equation*}
        \norm{\Hc_{\delta}(x) - \mathrm{Diag}(\nabla^2 f(x))} \leq M_f \delta. 
    \end{equation*}
    where the operator $\mathrm{Diag}(A)$ sets all the entries off the diagonal to be $0$ for a matrix $A$.
\end{lem}
\begin{proof}
    For any $x, d \in \Rn$, it follows Assumption \ref{asp:function} that 
    \begin{equation}
        \label{Eq_Le_diff_approx_0}
        f(x + d) \leq f(x) + \inner{\nabla f(x), d} + \frac{1}{2} \inner{\nabla f(x) d, d} + \frac{M_f}{6} \norm{d}^3. 
    \end{equation}
    By substituting \eqref{Eq_Le_diff_approx_0} into the expression of $\Gc_{\delta}$ and $\Hc_{\delta}$, we can achieve the desired result.  
\end{proof}

\begin{lem} \label{le:des-h}
    Suppose Assumption \ref{asp:function} holds. Then, for any $x \in \Rn$, $y \in \Rn$,  $\delta > 0$ and $\sigma > 0$, it holds that 
    \begin{equation*}
        h(y) \leq l(y; x, \delta, \sigma) + \dfrac{1}{\sigma} M_f^2 \delta^4 + \frac{1}{2} \left( \sup_{x \in \Rn} \norm{\Hc_{\delta} (x)} + L_f - \frac{1}{2} \sigma \right)\norm{y-x}^2.
    \end{equation*}
\end{lem}
\begin{proof}
    It directly follows from Lemma \ref{Le_diff_approx} that 
    \begin{equation*}
        \begin{aligned}
        h(y) \leq{} & f(x) + \inner{y-x, \nabla f(x)} + \frac{L_f}{2} \norm{y-x}^2 + r(y)\\
        \leq{} & f(x) + \inner{y-x,  \Gc_{\delta}(x)} + \norm{y-x}\norm{\Gc_{\delta}(x) - \nabla f(x)} + \frac{L_f}{2} \norm{y-x}^2 + r(y)\\
        \leq{} & l(y; x, \delta, \sigma) + \frac{1}{2} M_f\delta^2 \norm{y-x} + \frac{1}{2} \left( \sup_{x \in \Rn}\norm{\Hc_{\delta} (x)} + L_f - \sigma \right)\norm{y-x}^2\\
        \leq{} & l(y; x, \delta, \sigma) + \dfrac{1}{\sigma} M_f^2 \delta^4 + \frac{1}{2} \left( \sup_{x \in \Rn} \norm{\Hc_{\delta} (x)} + L_f - \frac{1}{2} \sigma \right)\norm{y-x}^2 . 
        \end{aligned}
    \end{equation*}
    The proof is completed.
\end{proof}

\begin{prop}
    Suppose Assumption \ref{asp:function} holds. 
    Then, for any $x\in \Rn$ and $\delta > 0$, we have
    \begin{equation} \label{eq:des-bri}
        \inf_{y \in \Rn} l(y; x, \delta,\sigma) \leq h(x) - \frac{1}{4} \gamma \norm{\Pc_{\gamma}(x)}^2 + \frac{1}{4} \gamma M_f^2 \delta^4,
    \end{equation}
    where 
    \begin{equation*}
        \sigma \geq 2\left(\sup_{x \in \Rn} \norm{\Hc_{\delta} (x)} + L_f\right),
    \end{equation*}
    and
    \begin{equation*}
        0 < \gamma \leq \frac{1}{\sigma + \sup_{x \in \Rn} \norm{\Hc_{\delta} (x)}}.
    \end{equation*}
\end{prop}

\begin{proof}
    For any $\gamma \leq (0, 1 / (\sigma + \norm{\Hc_{\delta} (x)}))$, $l(y; x, \delta, \sigma)$ is majorized by the following function,
    \begin{equation*}
        \tilde{l}(y) := f(x) + \inner{y - x, \Gc_{\delta}(x)} + r(y) + \frac{1}{2\gamma} \norm{y - x}^2.
    \end{equation*}
    Moreover, for any $z \in \mathop{\arg\min}_{y \in \Rn} \inner{y - x, \Gc_{\delta}(x)} + r(y) + \frac{1}{2\gamma} \norm{y - x}^2$, it holds that 
    \begin{equation}
        \tilde{l}(z) - \tilde{l}(x) \leq \inner{z - x, \Gc_{\delta}(x)} + r(z) - r(x) \leq - \frac{1}{2\gamma} \norm{z-x}^2,
    \end{equation}
    which implies that 
    \begin{equation*}
        \inf_{y \in \Rn} l(y; x, \delta, \sigma) \leq l(x; x, \delta, \sigma) - \frac{1}{2\gamma} \norm{x - \prox_{\gamma r}(x - \gamma \Gc_{\delta}(x))}^2.
    \end{equation*}
    Finally, it follows from Lemma \ref{Le_diff_approx} that
    \begin{equation*}
        \norm{\Pc_{\gamma} (x)} \leq \norm{x - \prox_{\gamma r}(x - \Gc_{\delta}(x))} + \frac{1}{2} \gamma M_f \delta^2.
    \end{equation*}
    Then we can conclude that 
    \begin{equation*}
        \frac{1}{\gamma} \norm{x - \prox_{\gamma r}(x - \gamma\Gc_{\delta}(x))}^2 
        \geq \frac{1}{2} \gamma \norm{\Pc_{\gamma} (x)}^2 - \frac{1}{2} \gamma M_f^2 \delta^4.
    \end{equation*}
    This completes the proof. 
\end{proof}

\begin{lem} \label{le:des-h-xk}
    Suppose Assumption \ref{asp:function} and Assumption \ref{asp:parameter} hold, and $\sigma_k \geq 2(L_H + L_f)$. 
    Then it holds that 
    \begin{equation*}
        h(\xkp) - h(\xk) \leq - \frac{1}{4} \gamma \norm{\Pc_{\gamma} (\xk)}^2 +  2 M_f^2 \delta_k^4 / \sigma_k + \varepsilon_k,
    \end{equation*}
    where $\gamma \in (0, 1 / (\sigma_k + L_H))$.
\end{lem}

\begin{proof}
    Taking $x = \xk$, $\delta = \delta_k$ and $\sigma = \sigma_k$ in \eqref{eq:des-bri} directly yields that
    \begin{equation*}
        \inf_{y \in \Rn} l(y; \xk, \delta_k, \sigma_k) 
        \leq h(\xk) - \frac{1}{4} \gamma \norm{\Pc_{\gamma}(\xk)}^2 + \frac{1}{4} \gamma M_f^2 \delta_k^4,
    \end{equation*}
    which together with the inexact condition \eqref{eq:inexact} implies that
    \begin{equation*}
        l(\xkp; \xk, \delta_k, \sigma_k) 
        \leq h(\xk) - \frac{1}{4} \gamma \norm{\Pc_{\gamma}(\xk)}^2 + \frac{1}{4} \gamma M_f^2 \delta_k^4 + \varepsilon_k.
    \end{equation*}
    Then, according to Lemma \ref{le:des-h}, we have
    \begin{equation*}
    \begin{aligned}
        h(\xkp) 
        \leq {} & l(\xkp; \xk, \delta_k, \sigma_k) + \dfrac{1}{\sigma_k} M_f^2 \delta_k^4 + \frac{1}{2} \left(L_H + L_f - \frac{1}{2} \sigma_k\right) \norm{\xkp - \xk}^2 \\
        \leq {} & h(\xk) - \frac{1}{4} \gamma \norm{\Pc_{\gamma}(\xk)}^2 + \frac{1}{4} \gamma M_f^2 \delta_k^4 + \dfrac{1}{\sigma_k} M_f^2 \delta_k^4 + \varepsilon_k \\
        & + \frac{1}{2} \left(L_H + L_f - \frac{1}{2} \sigma_k\right) \norm{\xkp - \xk}^2 \\
        \leq {} & h(\xk) - \frac{1}{4} \gamma \norm{\Pc_{\gamma}(\xk)}^2 + 2 M_f^2 \delta_k^4 / \sigma_k + \varepsilon_k,
    \end{aligned}
    \end{equation*}
    where the last inequality follows from the conditions $\gamma < 1 / (\sigma_k + L_H) < 1 / \sigma_k$ and $\sigma_k \geq 2(L_H + L_f)$.
    We complete the proof.
\end{proof}

\begin{coro} \label{cor:sum}
    Let Assumption \ref{asp:function} and Assumption \ref{asp:parameter} hold. 
    Suppose there exists $\sigma_{\max} > 2(L_f + L_H)$ such that $\sigma_k \in (2(L_f + L_H), \sigma_{\max})$ for any $k \geq 0$.  
    Then we have
    \begin{equation*}
        \frac{1}{4 N} \sum_{k = 0}^{N-1} \norm{\Pc_{\gamma} (\xk)}^2 \leq \frac{h(x_0) - h(x_{N-1}) + \sum_{k = 0}^{N-1} (2 M_f^2 \delta_k^4 / \sigma_k + \varepsilon_k)}{\gamma N},
    \end{equation*}
    where $\gamma \in \left( 0, 1 / (\sigma_{\max} + L_H) \right)$.
\end{coro}

\begin{proof}
    This is a direct consequence of Lemma \ref{le:des-h-xk}, and we omit the proof here.
\end{proof}


\begin{theo}
    Let Assumption \ref{asp:function} and Assumption \ref{asp:parameter} hold. 
    Suppose there exists $\sigma_{\max} > 2(L_f + L_H)$ such that $\sigma_k \in (2(L_f + L_H), \sigma_{\max})$ for any $k \geq 0$.  
    Moreover, we assume that 
    \begin{equation*}
        L_c := \sum_{k = 0}^{\infty} (\delta_k^4 + \varepsilon_k) < \infty. 
    \end{equation*}
    Then, given any  $\epsilon > 0$, Algorithm \ref{alg:IPZOPM} will return an $\epsilon$-stationary point of \eqref{opt:dfo} (see Definition \ref{def:stationary}) with
    \begin{equation*}
        \gamma \in \left( 0, \frac{1}{\sigma_{\max} + L_H}  \right)
    \end{equation*}
    in at most
    \begin{equation*}
        N \geq \frac{4 (h(x_0) - \underline{h} + C)}{\gamma \epsilon^2}
    \end{equation*}
    iterations, where $C = (M_f^2 / (L_f + L_H) + 1) L_c > 0$ is a constant.
\end{theo}

\begin{proof}
    To begin with, it can be readily verified that
    \begin{equation*}
    \begin{aligned}
        \sum_{k = 0}^{N - 1} (2 M_f^2 \delta_k^4 / \sigma_k + \varepsilon_k)
        \leq {} & \sum_{k = 0}^{N - 1} \left( \frac{M_f^2 \delta_k^4}{L_f + L_H} + \varepsilon_k \right) \\
        \leq {} & \left(\frac{M_f^2}{L_f + L_H} + 1 \right) \sum_{k = 0}^{N - 1} \left( \delta_k^4 + \varepsilon_k 
        \right) \\
        \leq {} & \left(\frac{M_f^2}{L_f + L_H} + 1 \right) L_c 
        = C.
    \end{aligned}
    \end{equation*}
    According to Corollary \ref{cor:sum}, we have
    \begin{equation*}
    \begin{aligned}
        \min_{k = 0, 1, \dotsc, N - 1} \norm{\Pc_{\gamma} (\xk)}^2
        \leq {} & \frac{1}{N} \sum_{k = 0}^{N-1} \norm{\Pc_{\gamma} (\xk)}^2 \\
        \leq {} & \frac{4(h(x_0) - h(x_{N-1}) + \sum_{k = 0}^{N-1} (2 M_f^2 \delta_k^4 / \sigma_k + \varepsilon_k))}{\gamma N} \\
        \leq {} & \frac{4(h(x_0) - \underline{h} + C)}{\gamma N}.
    \end{aligned}
    \end{equation*}
    Therefore, in at most 
    \begin{equation*}
        N \geq \frac{4 (h(x_0) - \underline{h} + C)}{\gamma \epsilon^2}
    \end{equation*}
    iterations, the iterate sequence $\{\xk\}_{k = 0}^{N - 1}$ generated by Algorithm \ref{alg:IPZOPM} will satisfy the following condition,
    \begin{equation*}
        \min_{k = 0, 1, \dotsc, N - 1} \norm{\Pc_{\gamma} (\xk)} \leq \epsilon.
    \end{equation*}
    The proof is completed.
\end{proof}


The global sub-linear convergence rate in the above theorem guarantees that IPZOPM is able to find an $\epsilon$-stationary point in at most $\ca{O}\left( \epsilon^{-2} \right)$ iterations. 
Since IPZOPM only evaluates the function values of $f$ on $2n$ points per iteration, the total number of function-value evaluations required to obtain an $\epsilon$-stationary point is $\ca{O}\left( n{\epsilon^{-2}} \right)$ at the most.

\section{Numerical Experiments}
\label{sec:numerical}

Comprehensive numerical experiments are conducted in this section to evaluate the numerical performance of IPZOPM. 
We use the Python language to implement the tested algorithms. 
And the corresponding experiments are conducted on a workstation with two Intel Xeon Gold 6242R CPU processors (at $3.10\,\mbox{GHz} \times 20 \times 2$) and $510\,$GB of RAM under Ubuntu 20.04.

\subsection{Implementation details}

For our algorithm IPZOPM, the parameter $\sigma_k$ is updated by the following heuristic strategy,
\begin{equation*}
    \sigma_k = 5000 \norm{x_k - x_{k - 1}}.
\end{equation*}
And we set 
\begin{equation*}
    \delta_k = \dfrac{1}{\sqrt{k + 1}}.
\end{equation*}
In the following experiments, the nonsmooth term $r$ is always the $\ell_1$ regularizer.
Therefore, the subproblem \eqref{opt:prox_sub} has a closed-form solution as mentioned in Subsection \ref{subsec:algorithm}.
The initial guess $x_0$ is randomly generated from the standard normal distribution.
And we terminate the tested algorithms if the following condition holds,
\begin{equation*}
    \norm{h (x_k) - h (x_{k - 1})} < 10^{-3},
\end{equation*}
or the maximum iteration number 1000 is reached.

\subsection{Comparison on LASSO problems}

We begin by performing the numerical comparison between IPZOPM and ZOPG \cite{pougkakiotis2022zeroth} on the following LASSO problem \cite{tibshirani1996regression},
\begin{equation}
  \min_{x \in \Rn} \,\frac{1}{2}\norm{A x - b}^2 + \mu \norm{x}_1,
\end{equation}
where $A \in \Rmn$ and $b \in \Rm$ are given data, and $\mu > 0$ is a constant. 
The above optimization model has been used extensively in high-dimensional statistics and machine learning.
For our testing, we first generate the matrix $A \in \Rmn$ and two vectors $u \in \Rn$ and $l \in \Rm$ randomly from the standard normal distribution. Then we set $b = Au + \sqrt{0.001} \cdot l$.


Figure \ref{fig:lasso} depicts the decay of function values versus the iteration numbers for four different cases.
We can observe that IPZOPM exhibits a faster convergence rate compared to ZOPG.

\begin{figure}[H]
	\centering
	
	\subfigure[$n = 1000, m = 100$]{
		\label{subfig:lasso_1}
		\includegraphics[width=0.4\linewidth]{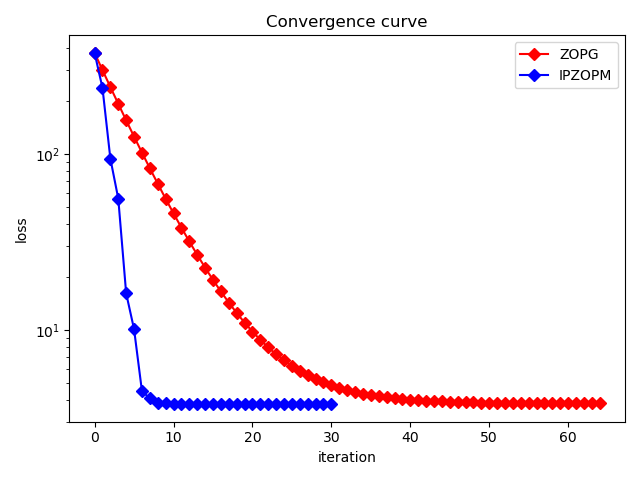}
	}
	\subfigure[$n = 2000, m = 200$]{
		\label{subfig:lasso_2}
		\includegraphics[width=0.4\linewidth]{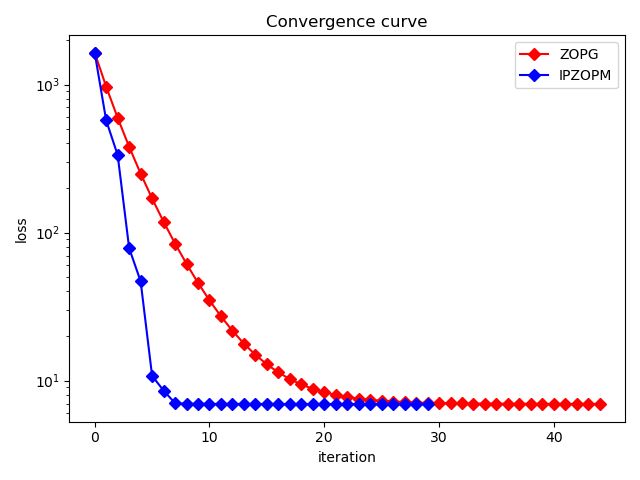}
	}

    \subfigure[$n = 3000, m = 300$]{
		\label{subfig:lasso_3}
		\includegraphics[width=0.4\linewidth]{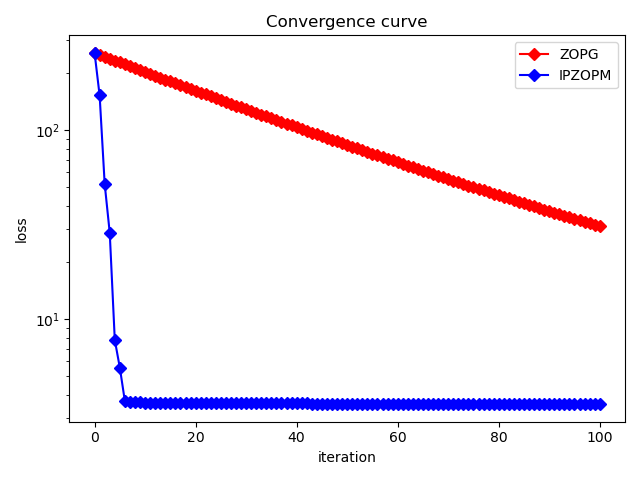}
	}
     \subfigure[$n = 4000, m = 400$]{
		\label{subfig:lasso_4}
		\includegraphics[width=0.4\linewidth]{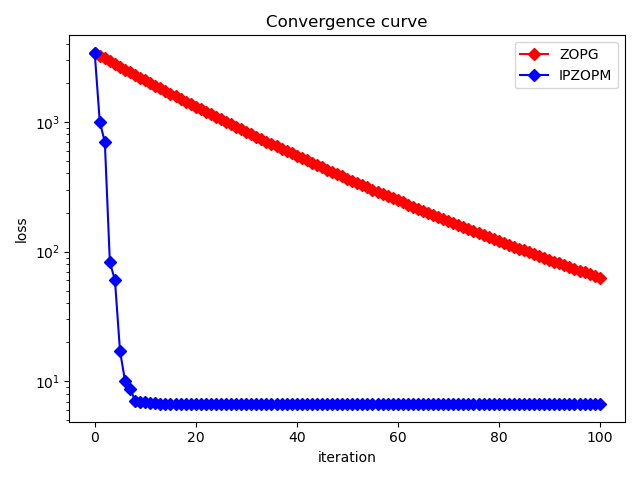}
	}
	
	\caption{Comparison between IPZOPM and ZOPG in solving LASSO problem.}
	\label{fig:lasso}
\end{figure}

\subsection{Comparison on binary classification problems}

The next experiment aims to assess the numerical performance of IPZOPM and ZOPG on the binary classification \cite{huang2019faster} problems.
Specifically, given a set of samples $\{a_i, l_i\}_{i = 1}^m$ with $a_i \in \Rn$ and $l_i \in \{-1, 1\}$, this type of problem identifies an optimal classifier $x \in \Rn$ by solving the following optimization model,
\begin{equation}
  \min_{x \in \Rn} \frac{1}{m}\sum _{i=1}^m f_i (x) + \lambda \norm{x}^2 + \mu \norm{x}_1,
\end{equation}
where $f_i(x)$ is a smooth loss function that only returns the function value given an input. 
Here, we specify the following nonconvex sigmoid loss function,
\begin{equation}
 f_i(x) = \frac{1}{1 + \exp(l_i a_i^{\top}x)},
\end{equation}
in the zeroth-order setting.
In addition, $\lambda > 0$ and $\mu > 0$ are two constants.

Four real-world datasets that are publicly available online\footnote{\url{https://www.csie.ntu.edu.tw/~cjlin/libsvmtools/datasets/binary.html}} are tested in this experiment, including a4a, a9a, w4a, and w8a.
We summarize the numbers of features $(n)$ and samples $(m)$ for each dataset in Table \ref{tb:data}.
Moreover, we fix $\lambda = \mu = 10^{-3}$.
The corresponding numerical results are shown in Figure \ref{fig:binary}.
It can be observed that IPZOPM consistently achieves a lower function value than ZOPG with the same number of iterations,
which demonstrates the superiority of our algorithm.

\begin{table}[htbp]
\caption{Real datasets tested in the binary classification problems.}
\label{tb:data}
\centering	
\begin{tabular}{c|c|c}
	\hline
	~~~datasets~~~ & ~~~$\#$samples $(m)$~~~ & ~~~$\#$features $(n)$~~~ \\ \hline
	a4a & 4781   & 122     \\ \hline
	a9a & 32561   & 123    \\ \hline
	w4a & 7366   & 300     \\ \hline
	w8a & 49749   & 300    \\ \hline
\end{tabular}
\end{table}

\begin{figure}[H]
	\centering
	
	\subfigure[a4a]{
		\label{subfig:binary_1}
		\includegraphics[width=0.4\linewidth]{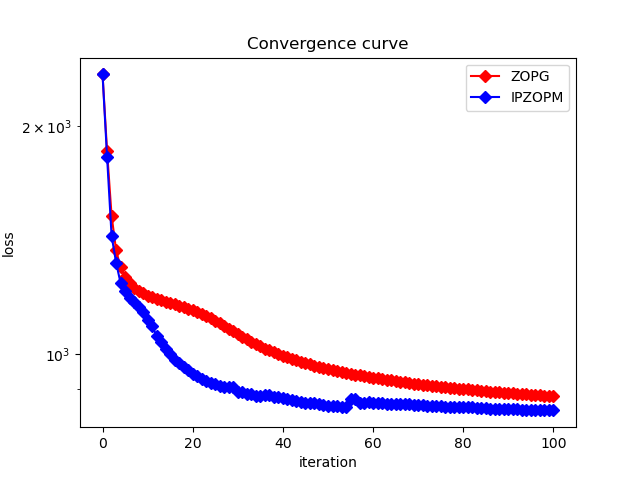}
	}
	\subfigure[a9a]{
		\label{subfig:binary_2}
		\includegraphics[width=0.4\linewidth]{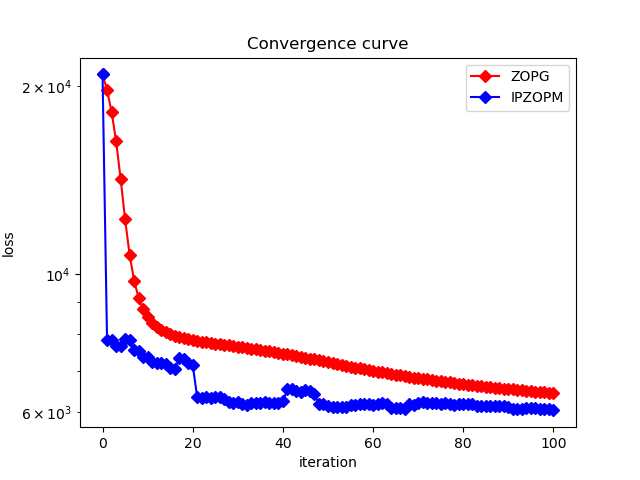}
	}

    \subfigure[w4a]{
		\label{subfig:binary_3}
		\includegraphics[width=0.4\linewidth]{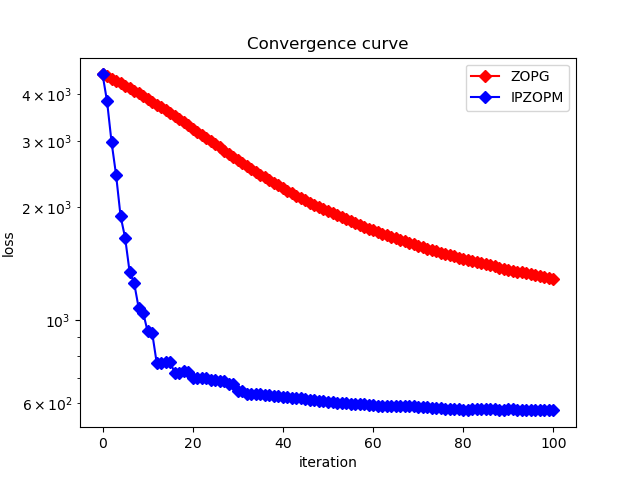}
	}
    \subfigure[w8a]{
		\label{subfig:binary_4}
		\includegraphics[width=0.4\linewidth]{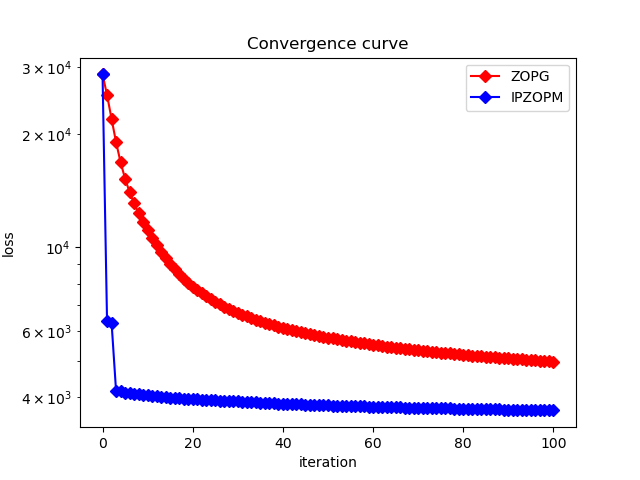}
	}
	
	\caption{Comparison between IPZOPM and ZOPG in solving binary classification problems.}
	\label{fig:binary}
\end{figure}

\section{Conclusion}
\label{sec:conclusion}

In this paper, we propose a novel method named IPZOPM to solve the composite optimization problem \eqref{opt:dfo} under the zeroth-order setting.
The proposed method estimates the diagonal of the Hessian matrix as a preconditioner for the proximal gradient method, resulting in a significant improvement in its performance in practical scenarios.
Most notably, this scheme does not require additional evaluations of function values per iteration.
We also establish a global convergence guarantee to stationary points and a worst-case complexity $\ca{O} (n\epsilon^{-2})$ under mild conditions.
Numerical experiments illustrate the promising potential of IPZOPM in large-scale applications.

Finally, we mention two related topics worthy of future studies. 
One is the possibility of developing zeroth-order quasi-Newton approaches to further reduce the per-iteration complexity. 
Another is to extend the framework of IPZOPM to Riemannian manifolds so that a wider range of applications, such as sparse PCA \cite{xiao2021exact,wang2023communication}, can benefit from our strategy.

\bmhead{Acknowledgement}

The work of Xin Liu was supported in part by the National Natural Science Foundation of China (12125108, 12226008, 12021001, 12288201, 11991021), Key Research Program of Frontier Sciences, Chinese Academy of Sciences (ZDBS-LY-7022), and CAS AMSS-PolyU Joint Laboratory of Applied Mathematics.

\backmatter

\bibliography{sn-bibliography}

\end{document}